\documentclass{article}
\usepackage{amssymb}
\usepackage{amsmath}
\usepackage{amsthm}
\usepackage{setspace}
\usepackage{braket}

\newtheorem{theorem}{Theorem}[section]
\newtheorem{lemma}[theorem]{Lemma}
\newtheorem{proposition}[theorem]{Proposition}
\newtheorem{corollary}[theorem]{Corollary}

\theoremstyle{definition}

\newtheorem{example}[theorem]{Example}
\newtheorem{algorithm}[theorem]{Algorithm}

\theoremstyle{remark}
\newtheorem{remark}[theorem]{Remark}

\numberwithin{equation}{section}

\begin{document}

\title{A combinatorial problem and numerical semigroups}

\author{Aureliano M. Robles-P\'erez\thanks{Both of the authors are supported by FQM-343 (Junta de Andaluc\'{\i}a), MTM2014-55367-P (MINECO, Spain), and FEDER funds. The second author is also partially supported by Junta de Andaluc\'{\i}a/FEDER Grant Number FQM-5849.} \thanks{Departamento de Matem\'atica Aplicada, Universidad de Granada, 18071-Granada, Spain. \newline E-mail: \textbf{arobles@ugr.es}}
	\mbox{ and} Jos\'e Carlos Rosales$^*$\thanks{Departamento de \'Algebra, Universidad de Granada, 18071-Granada, Spain. \newline E-mail: \textbf{jrosales@ugr.es}} }

\date{ }

\maketitle

\begin{abstract}
Let $a=(a_1,\ldots,a_n)$ and $b=(b_1,\ldots,b_n)$ be two $n$-tuples of positive integers, let $X$ be a set of positive integers, and let $g$ be a positive integer. In this work we show an algorithmic process in order to compute all the sets $C$ of positive integers that fulfill the following conditions:
\begin{enumerate}
	\item the cardinality of $C$ is equal to $g$;
	\item if $x,y\in \mathbb{N} \setminus \{0\}$ and $x+y\in C$, then $C \cap \{x,y\} \neq \emptyset$;
	\item if $x \in C$ and $\frac{x-b_i}{a_i} \in \mathbb{N} \setminus \{0\}$ for some $i\in \{1,\ldots,n\}$, then $\frac{x-b_i}{a_i} \in C$;
	\item $X \cap C = \emptyset$. 
\end{enumerate}	
	
\end{abstract}

\medskip


\section{Introduction}

In certain lottery game, we have to guess six (positive) numbers to win the main prize. By analysing the results of different draws, it has been observed that some patterns happen frequently. In effect, if $C$ is a winner combination, then is rather probable that one of the following conditions is fulfilled.
\begin{itemize}
	\item[] \hspace{-20pt} $(C1)$ If $x,y$ are positive integers such that $x+y\in C$, then $C \cap \{x,y\} \neq \emptyset$.
	\item[] \hspace{-20pt} $(C2)$ If $x \in C$ and $x-4$ is a positive integer, then $x-4 \in C$.
	\item[] \hspace{-20pt} $(C3)$ If $x \in C$ and $\frac{x-1}{2}$ is a positive integer, then $\frac{x-1}{2} \in C$.
	\item[] \hspace{-20pt} $(C4)$ $5 \not\in C$.
\end{itemize}
The question naturally arises as to whether a combination fulfilling the four conditions simultaneously exists. The purpose of this work will be to give an answer to this type of combinatorial problems.

Let $\mathbb{Z}$ and $\mathbb{N}$ be the sets of integers and non-negative integers, respectively. Let $a=(a_1,\ldots,a_n)$ and $b=(b_1,\ldots,b_n)$ be two $n$-tuples (with $n\geq 1$) of positive integers, let $X$ be a non-empty subset of $\mathbb{N} \setminus \{0\}$, and let $g$ be a positive integer. Let us denote by $\mathrm{P}(a,b,X,g)$ the problem of computing all the subsets $C$ of $\mathbb{N} \setminus \{0\}$ that fulfill the following conditions.
\begin{itemize}
	\item[] \hspace{-20pt} $(P1)$ The cardinality of $C$ is equal to $g$.
	\item[] \hspace{-20pt} $(P2)$ If $x,y\in \mathbb{N} \setminus \{0\}$ and $x+y\in C$, then $C \cap \{x,y\} \neq \emptyset$.
	\item[] \hspace{-20pt} $(P3)$ If $x \in C$ and $\frac{x-b_i}{a_i} \in \mathbb{N} \setminus \{0\}$ for some $i\in \{1,\ldots,n\}$, then $\frac{x-b_i}{a_i} \in C$.
	\item[] \hspace{-20pt} $(P4)$ $X \cap C = \emptyset$. 
\end{itemize}

With the previous notation, we observe that the problem proposed at the beginning is just $\mathrm{P}((1,2),(4,1),\{5\},6)$.

A \emph{numerical semigroup} (see \cite{springer}) is a submonoid $S$ of $(\mathbb{N},+)$ such that $\mathbb{N} \setminus S$ is finite. The cardinality of $\mathbb{N} \setminus S$ is so-called the \emph{genus} of $S$ and is denoted by $\mathrm{g}(S)$.

It is easy to see that $C$ is a solution of $\mathrm{P}(a,b,X,g)$ if and only if $S= \mathbb{N} \setminus C$ is a numerical semigroup that fulfills the following conditions.
\begin{itemize}
	\item[] \hspace{-20pt} $(S1)$ $\mathrm{g}(S)=g$.
	\item[] \hspace{-20pt} $(S2)$ If $s \in S \setminus \{0\}$, then $as+b \in S^n$ (where $as+b=(a_1s+b_1,\ldots,a_ns+b_n))$.
	\item[] \hspace{-20pt} $(S3)$ $X \subseteq S$. 
\end{itemize}

Let us denote by $\mathcal{N}(a,b,X)$ the set
$$\left\{ S \mid S \; \mbox{is a numerical semigroup}, \; X \subseteq S, \; \mbox{and} \; as+b\in S^n \; \mbox{for all} \; s \in S \setminus \{0\} \right\}.$$
With this notation, the solutions of $\mathrm{P}(a,b,X,g)$ are the elements of the set $\left\{ \mathbb{N} \setminus S \mid S \in \mathcal{N}(a,b,X) \; \mbox{and} \; \mathrm{g}(S)=g \right\}$.

Let $S$ be a numerical semigroup. The \emph{Frobenius number} of $S$ (see \cite{alfonsin}), denoted by $\mathrm{F}(S)$, is the maximum integer that does not belong to $S$.

A \emph{Frobenius variety} (see \cite{variedades}) is a non-empty family of numerical semigroups $\mathcal{V}$ that fulfills the following conditions.
\begin{enumerate}
	\item[] \hspace{-20pt} $(V1)$ If $S,T \in \mathcal{V}$, then $S\cap T\in \mathcal{V}$.
	\item[] \hspace{-20pt} $(V2)$ If $ S\in \mathcal{V}$ and $S\neq \mathbb{N}$, then $S\cup \{\mathrm{F}(S)\} \in \mathcal{V}$.
\end{enumerate}

In Section~\ref{monoids} we will see that $\mathcal{N}(a,b,X)$ is a Frobenius variety. In addition, we will show that such a variety is finite if and only if $\gcd(X \cup \{b_1,\ldots,b_n\})=1$ (where, as usual, $\gcd(A)$ is the \emph{greatest common divisor} of the elements in $A$).

Let us denote by $M(a,b,X)$ the intersection of all the elements in $\mathcal{N}(a,b,X)$. Observe that $M(a,b,X)$ is always a submonoid of $(\mathbb{N},+)$. In addition, we will prove that $M(a,b,X)$ is a numerical semigroup if and only if $\mathcal{N}(a,b,X)$ has finitely many elements.

In Section~\ref{monoids} we will show that $\mathrm{P}(a,b,X,g)$ has solution if and only if the cardinality of $\mathbb{N} \setminus M(a,b,X)$ is greatest than or equal to $g$. Moreover, we will give an algorithm in order to compute $M(a,b,X)$. Therefore, we will have an algorithmic process to decide whether $\mathrm{P}(a,b,X,g)$ has solution.

In Section~\ref{tree}, with the help of some results from \cite{variedades}, we will arrange the elements of $\mathcal{N}(a,b,X)$ in a tree with root $\mathbb{N}$. Moreover, we will characterize the children of a vertex in such a tree and, consequently, will have a recursive procedure in order to build $\mathcal{N}(a,b,X)$. Accordingly, we will have an algorithmic process to compute all the elements of $\mathcal{N}(a,b,X)$ with a fixed genus and, in particular, an algorithm to compute all the solutions of $\mathrm{P}(a,b,X,g)$.

Finally, in Section~\ref{generalization} we will state and solve a generalization of the problem $\mathrm{P}(a,b,X,g)$.

\section{$\mathbf{(a,b)}$-monoids}\label{monoids}

In this work, unless stated otherwise, $a=(a_1,\ldots,a_n)$ and $b=(b_1,\ldots,b_n)$ denote two $n$-tuples of positive integers. If $X \subseteq \mathbb{N}$, then $\mathcal{N}(a,b,X)$ is the set $\left\{ S \, \mid S \;\, \mbox{is a numerical semigroup}, \; X \subseteq S, \; \mbox{and} \; as+b\in S^n \;\, \mbox{for all} \; s \in S \setminus \{0\} \right\},$ \\
with $as+b = (a_1s+b_1,\ldots,a_ns+b_n)$.

\begin{proposition}\label{prop1}
	$\mathcal{N}(a,b,X)$ is a Frobenius variety.
\end{proposition}

\begin{proof}
	It is clear that $\mathbb{N} \in \mathcal{N}(a,b,X)$ and, therefore, $\mathcal{N}(a,b,X) \neq \emptyset$. It is also clear that, if $S,T \in \mathcal{N}(a,b,X)$, then $S\cap T \in \mathcal{N}(a,b,X)$. Now, let $S\in \mathcal{N}(a,b,X)$ such that $S \neq \mathbb{N}$. In order to show that $S \cup \{\mathrm{F}(S)\} \in \mathcal{N}(a,b,X)$, it is enough to see that $a\mathrm{F}(S)+b \in (S \cup \{\mathrm{F}(S)\})^n$. Observe that, if $i \in \{1,\ldots,n\}$, then $a_i\mathrm{F}(S)+b_i > \mathrm{F}(S)$ and, therefore, $a_i\mathrm{F}(S)+b_i \in S \cup \{\mathrm{F}(S)\}$. Consequently, $a\mathrm{F}(S)+b \in (S \cup \{\mathrm{F}(S)\})^n$.
\end{proof}

Let $M$ be a submonoid of $(\mathbb{N},+)$. We will say that $M$ is an \emph{$(a,b)$-monoid} if $am+b \in M^n$ for all $m \in M \setminus \{0\}$.

\begin{proposition}\label{prop2}
	Let $M$ be a submonoid of $(\mathbb{N},+)$ and let $X \subseteq {\mathbb N}$. Then $M$ is an $(a,b)$-monoid if and only if $M$ is the intersection of elements in $\mathcal{N}(a,b,X)$.
\end{proposition}

\begin{proof}
	The sufficient condition is trivial. For the necessary one, let us take $M_k = M \cup \{k,\to\}$, for all $k \in \mathbb{N}$ (where $\{k,\to\}=\{n\in \mathbb{N} \mid n \geq k\})$. Then it is clear that $M_k \in \mathcal{N}(a,b,X)$ and that $M=\bigcap_{k\in \mathbb{N}} M_k$.
\end{proof}

Let us observe that, if we denote by $M(a,b,X)= \bigcap_{S\in \mathcal{N}(a,b,X)} S$, then $M(a,b,X)$ is the smallest $(a,b)$-monoid containing $X$.

\begin{theorem}\label{thm3}
	Let $X$ be a non-empty subset of $\mathbb{N} \setminus \{0\}$ and let $g$ be a positive integer. Then the problem $\mathrm{P}(a,b,X,g)$ has solution if and only if the cardinality of $\mathbb{N} \setminus M(a,b,X)$ is greatest than or equal to $g$.
\end{theorem}

\begin{proof}
	\emph{(Necessity.)} If $C$ is a solution of $\mathrm{P}(a,b,X,g)$ and we take $S=\mathbb{N} \setminus C$, then $S \in \mathcal{N}(a,b,X)$ and $\mathrm{g}(S)=g$. Since $M(a,b,X) \subseteq S$, we have that $\mathbb{N} \setminus S \subseteq \mathbb{N} \setminus M(a,b,X)$ and, thereby, the cardinality of $\mathbb{N} \setminus M(a,b,X)$ is greatest than or equal to $g$.
	
	\emph{(Sufficiency.)} Let us suppose that $\mathbb{N} \setminus M(a,b,X)=\{c_1<\cdots<c_g<\cdots\}$. If we take $S= M(a,b,X) \cup \{c_g+1,\to\}$, then it is clear that $S \in \mathcal{N}(a,b,X)$ and $\mathrm{g}(S)=g$. Therefore, $C=\mathbb{N} \setminus S$ is a solution of $\mathrm{P}(a,b,X,g)$.
\end{proof}

Let us observe that, if $\mathrm{P}(a,b,X,g)$ has solution and, in addition, we know $M(a,b,X)$, then the proof of the sufficient condition in the previous theorem gives us a method to compute a solution of $\mathrm{P}(a,b,X,g)$.

If $X$ is a non-empty subset of $\mathbb{N}$, then we denote by $\langle X \rangle$ the submomoid of $(\mathbb{N},+)$ generated by $X$, that is,
$$\langle X \rangle = \{ \lambda_1x_1+\cdots+\lambda_kx_k \mid k \in \mathbb{N}\setminus \{0\}, \; x_1,\ldots,x_k \in X, \; \mbox{and} \; \lambda_1,\ldots,\lambda_k \in \mathbb{N} \}.$$
If $M=\langle X \rangle$, then we will say that $M$ is generated by $X$ or, equivalently, that $X$ is a \emph{system of generators} of $M$. The next result is well known (see, for instance, \cite{springer}).

\begin{lemma}\label{lem4}
	If $X \subseteq \mathbb{N}$, then $\langle X \rangle$ is a numerical semigroup if and only if $\gcd(X)=1$.
\end{lemma}

We know that $M(a,b,X)$ is a submonoid of $(\mathbb{N},+)$. From the following proposition we will get that, if $X \subseteq \mathbb{N} \setminus \{0\}$, then $M(a,b,X)$ is a numerical semigroup if and only if $\gcd(X\cup \{b_1,\ldots,b_n\})=1$.

\begin{proposition}\label{prop5}
	If $X \subseteq \mathbb{N} \setminus \{0\}$, then $\gcd(M(a,b,X))=\gcd(X\cup \{b_1,\ldots,b_n\})$.
\end{proposition}

\begin{proof}
	Let $d=\gcd(M(a,b,X))$ and $d'=\gcd(X\cup \{b_1,\ldots,b_n\})$. In order to prove the proposition, we will show that $d'\mid d$ and $d\mid d'$. (As usual, if $p,q$ are positive integers, then $p \mid q$ means that $p$ divides $q$.)
	
	First of all, it is clear that $\langle \{d'\} \rangle$ is an $(a,b)$-monoid containing $X$. Therefore, $M(a,b,X) \subseteq \langle \{d'\} \rangle$ and, consequently, $d'\mid d$.
	
	Now, let us take $x \in X$. Then $\{x,a_1x+b_1,\ldots,a_nx+b_n\} \subseteq M(a,b,X)$ and $\gcd\{x,a_1x+b_1,\ldots,a_nx+b_n\}=\gcd\{x,b_1,\ldots,b_n\}$. Therefore, we have that $X \cup \left( \bigcup_{i=1}^n \{a_ix+b_i \mid x \in X\} \right) \subseteq M(a,b,X)$ and $\gcd\left( X \cup \left( \bigcup_{i=1}^n \{a_ix+b_i \mid x \in X\} \right) \right)$ $= \gcd(X\cup \{b_1,\ldots,b_n\})=d'$. Accordingly, $d\mid d'$.
\end{proof}

In the next result we show when the Frobenius variety $\mathcal{N}(a,b,X)$ is finite.

\begin{theorem}\label{thm6}
	Let $X$ be a subset of $\mathbb{N}\setminus \{0\}$. Then the following conditions are equivalent.
	\begin{enumerate}
		\item[\textit{1.}] $\mathcal{N}(a,b,X)$ is finite.
		\item[\textit{2.}] $M(a,b,X)$ is a numerical semigroup.
		\item[\textit{3.}] $\gcd(X\cup \{b_1,\ldots,b_n\})=1$.
	\end{enumerate}
\end{theorem}

\begin{proof}
	The equivalence between conditions \textit{2} and \textit{3} is consequence of Lemma~\ref{lem4} and Proposition~\ref{prop5}. Now, let us see the equivalence between conditions \textit{1} and \textit{2}.
	
	$\textit{(1.} \Rightarrow \textit{2.)}$ It is enough to observe that the finite intersection of numerical semigroups is another numerical semigroup.
	
	$\textit{(2.} \Rightarrow \textit{1.)}$ If $S \in \mathcal{N}(a,b,X)$, then $M(a,b,X) \subseteq S$. Thus, $S=M(a,b,X) \cup Y$ with $Y\subseteq \mathbb{N} \setminus M(a,b,X)$. Since $\mathbb{N} \setminus M(a,b,X)$ is finite, we conclude that $\mathcal{N}(a,b,X)$ is finite.
\end{proof}

Our next aim in this section will be to give an algorithm in order to compute $M(a,b,X)$. For this is fundamental the following result.

\begin{proposition}\label{prop7}
	Let $M$ be a submonoid of $(\mathbb{N},+)$ generated by $X \subseteq \mathbb{N}\setminus \{0\}$. Then $M$ is an $(a,b)$-monoid if and only if $ax+b \in M^n$ for all $x \in X$.
\end{proposition}

\begin{proof}
	The necessary condition is trivial. For the sufficiency, let $m \in M\setminus \{0\}$. Then there exist $x_1,\ldots,x_t \in X$ such that $m=x_1+\cdots+x_t$. If $t=1$, then $m=x_1$ and $am+b=ax_1+b \in M^n$. If $t\geq 2$, then $am+b=a(x_1+\cdots+x_t)+b=a(x_1+\cdots+x_{t-1})+ax_t+b$. Since $a(x_1+\cdots+x_{t-1}),\,ax_t+b \in M^n$, we finish the proof.
\end{proof}

The above proposition will be useful in order to determine whether a submonoid $M$ of $(\mathbb{N},+)$ is or is not an $(a,b)$-monoid. Let us see an example.

\begin{example}\label{exmp8}
	$S=\langle \{4,5,11\} \rangle$ is an $((1,2),(4,1))$-monoid because $(1,2)4+(4,1)=(8,9)\in S^2$, $(1,2)5+(4,1)=(9,11)\in S^2$, and $(1,2)11+(4,1)=(15,23)\in S^2$. Nevertheless, $T=\langle \{5,7,9\} \rangle$ is not an $((1,2),(4,1))$-monoid because $(1,2)5+(4,1)=(9,11)\notin T^2$ (observe that $11\notin T$).
\end{example}

With the help of Proposition~\ref{prop7}, it would be possible to give an algorithm in order to compute $M(a,b,X)$. However, we are going to postpone such an algorithm because, as we will see now, we can focus on case in which $\gcd(X \cup \{b_1,\ldots,b_n\})=1$, and thus simplify the computations.

We say that an integer $d$ divides an $n$-tupla of integers $c=(c_1,\ldots,c_n)$ if $d \mid c_i$ for all $i\in\{1,\ldots,n\}$. In such a case, we denote by $\frac{c}{d} = \left( \frac{c_1}{d},\ldots,\frac{c_n}{d} \right)$. If $A\subseteq \mathbb{Z}$ and $k\in \mathbb{Z}$, then $kA=\{ka \mid a \in A\}$. Finally, if $A\subseteq \mathbb{Z}$, $d\in \mathbb{Z}$, and $d\mid a$ for all $a\in A$, then $\frac{A}{d}=\left\{ \frac{a}{d} \mid a \in A \right\} $.

\begin{lemma}\label{lem9}
	Let $M$ be an $(a,b)$-monoid such that $M \not= \{0\}$. If $\gcd(M)=d$, then
	\begin{enumerate}
		\item $d$ divides $b$;
		\item if $d' \in \mathbb{N} \setminus \{0\}$ and $d'\mid d$, then $\frac{M}{d'}$ is an $\left(a,\frac{b}{d'}\right)$-monoid;
		\item if $k \in \mathbb{N} \setminus \{0\}$, then $kM$ is an $(a,kb)$-monoid.
	\end{enumerate}
\end{lemma}

\begin{proof}
	\textit{1.} If we take $X=M \setminus \{0\}$, and having in mind that $M(a,b,X)$ is the smallest $(a,b)$-monoid containing $X$, then this item is consequence of Proposition~\ref{prop5}.
	
	\textit{2.} It is clear that $\frac{M}{d'}$ is a submonoid of $(\mathbb{N},+)$. In addition, if $x \in \frac{M}{d'} \setminus \{0\}$, then $d'x \in M \setminus \{0\}$ and, therefore, $ad'x+b \in M^n$. Consequently, $ax+\frac{b}{d'} \in \frac{M^n}{d'} = \left(\frac{M}{d'}\right)^n$. 
	
	\textit{3.} It is clear that $kM$ is a submonoid of $(\mathbb{N},+)$. Now, arguing as in the previous item, if $x \in kM \setminus \{0\}$, then $\frac{x}{k} \in M \setminus \{0\}$ and, therefore, $a\frac{x}{k}+b \in M^n$. Consequently, $ax+kb \in k\cdot M^n = (kM)^n$.
\end{proof}

The next proposition says us that, in order to compute $M(a,b,X)$, it is sufficient to calculate $d=\gcd\left(X \cup \{b_1,\ldots,b_n\}\right)$ and $M\left(a,\frac{b}{d},\frac{X}{d}\right))$.  By the way, observe that $\gcd\left(\frac{X}{d} \cup \left\{ \frac{b_1}{d},\ldots,\frac{b_n}{d} \right\}\right)=1$ and, therefore, $M\left(a,\frac{b}{d},\frac{X}{d}\right)$ is a numerical semigroup.

\begin{proposition}\label{prop10}
	Let $X$ be a subset of $\mathbb{N} \setminus \{0\}$. If $\gcd\left(X \cup \{b_1,\ldots,b_n\}\right)=d$, then $M(a,b,X)=d\cdot M\left(a,\frac{b}{d},\frac{X}{d}\right)$.
\end{proposition}

\begin{proof}
	From item 3 of Lemma~\ref{lem9}, we have that $d\cdot M\left(a,\frac{b}{d},\frac{X}{d}\right)$ is an $(a,b)$-monoid containing $X$. Therefore, $M(a,b,X) \subseteq d\cdot M\left(a,\frac{b}{d},\frac{X}{d}\right)$.
	
	On the other hand, from Proposition~\ref{prop5} and item 2 of Lemma~\ref{lem9}, we deduce that $\frac{M(a,b,X)}{d}$ is an $\left(a,\frac{b}{d}\right)$-monoid containing $\frac{X}{d}$. Consequently, $M\left(a,\frac{b}{d},\frac{X}{d}\right) \subseteq \frac{M(a,b,X)}{d}$, that is, $d\cdot M\left(a,\frac{b}{d},\frac{X}{d}\right) \subseteq M(a,b,X)$.
\end{proof}

We are now ready to show the announced algorithm.

\begin{algorithm}\label{alg11}
	\mbox{ }
	
	\noindent INPUT: A non-empty finite set of positive integers $X$ such that \par \hspace{.73cm} $\gcd\left( X \cup \{b_1,\ldots,b_n\} \right) = 1$. \par
	\noindent OUTPUT: $M(a,b,X)$. \par
	(1) $A=\emptyset$ and $G=X$. \par
	(2) If $G \setminus A = \emptyset$, then return $\langle G \rangle$ and stop the algorithm.\par
	(3) $m=\min(G \setminus A)$. \par
	(4) $H=\left\{ a_im+b_i \mid i \in \{1,\ldots,n\} \mbox{ and } a_im+b_i\notin \langle G \rangle \right\}$.\par
	(5) If $H=\emptyset$, then go to (7). \par
	(6) $G = G \cup H$. \par
	(7) $A = A \cup \{m\}$ and go to (2).\par
\end{algorithm}

Let us justify briefly the performance of this algorithm. Let us observe that, if the algorithm stops, then it returns $\langle G \rangle$ such that $ag+b \in \langle G \rangle^n$ for all $g \in G$. Therefore, by applying Proposition~\ref{prop7}, we have that $\langle G \rangle$ is an $(a,b)$-monoid. In addition, by construction, it is clear that $G$ must be a subset of every $(a,b)$-monoid which contains $X$. Thus, $\langle G \rangle$ is the smallest $(a,b)$-monoid containing $X$. Consequently, in order to justify the algorithm, it will be enough to see that the algorithm stops. In fact, when we arrive to step (7) at the first time, we have that $\gcd(G)=1$ and, thereby, $\langle G \rangle$ is a numerical semigroup. Therefore, $\mathbb{N} \setminus \langle G \rangle$ is finite and we can go to the step (6) only in a finite number of times.

\begin{remark}\label{rem-X}
	If we suppose, for a moment, that $X=\emptyset$ (in some sense, we are removing condition $(P4)$ in $\mathrm{P}(a,b,X,g)$), then it is obvious that $S_k = \{0,k,\to\}$, for all $k \in \mathbb{N}$ (where $\{0,k,\to\}=\{0\} \cup  \{n\in \mathbb{N} \mid n \geq k\})$, are numerical semigroups that belong to $\mathcal{N}(a,b,X)$ independently of the chosen $n$-tuples $a,b$. Thus $M(a,b,X)=\{0\}$, that is, the submonoid $S$ of $(\mathbb{N},+)$ generated by $X=\emptyset$.  We can observed that, if we apply Algorithm~\ref{alg11} in this case, we obtain this output.
\end{remark}

\begin{remark}\label{rem-(a,b)}
	Now, let us suppose that $a,b$ are $0$-tuples, that is, we remove condition $(P3)$ in $\mathrm{P}(a,b,X,g)$. In this case, it is straightforward to see that $M(a,b,X)$ is just the monoid generated by $X$. In addition, that is the output of Algorithm~\ref{alg11} in this case.
\end{remark}

After Remarks~\ref{rem-X} and \ref{rem-(a,b)}, we can conclude that $(P2)$ is the essential condition in order to have a structure related with monoids and, specifically, with numerical semigroups.

Let us illustrate the performance of Algorithm~\ref{alg11} with two examples.

\begin{example}\label{exmp12}
	We are going to compute $M=M((1,2),(4,1),\{5\})$.
	\begin{itemize}
		\item $A=\emptyset$ and $G=\{5\}$.
		\item $m=5$, $H=\{9,11\}$, $G=\{5,9,11\}$, and $A=\{5\}$.
		\item $m=9$, $H=\{13\}$, $G=\{5,9,11,13\}$, and $A=\{5,9\}$.
		\item $m=11$, $H=\emptyset$, $G=\{5,9,11,13\}$, and $A=\{5,9,11\}$.
		\item $m=13$, $H=\{17\}$, $G=\{5,9,11,13,17\}$, and $A=\{5,9,11,13\}$.
		\item $m=17$, $H=\emptyset$, $G=\{5,9,11,13,17\}$, and $A=\{5,9,11,13,17\}$.
	\end{itemize}
	Therefore, $M=\langle \{5,9,11,13,17\} \rangle$.
	
	Going back to the problem $\mathrm{P}((1,2),(4,1),\{5\},6)$ of the introduction, we have that, since $\mathbb{N} \setminus M = \{1,2,3,4,6,7,8,12 \}$ has cardinality equal to 8, then Theorem~\ref{thm3} asserts that the proposed problem has solution. Moreover, the solutions will be some subsets, with cardinality equal to $6$, of $\{1,2,3,4,6,7,8,12 \}$. In addition, by the proof of the sufficiency of Theorem~\ref{thm3}, we know that $\{1,2,3,4,6,7\}$ is a solution of such a problem.
\end{example}

We finish this section with an example in which $M(a,b,X)$ is not a numerical semigroup.

\begin{example}\label{exmp13}
	Let us see that $\mathrm{P}\left( (2,3), (4,2), \{6,8\}, 9 \right)$ has solution. For that, we begin with the computation of $M\left((2,3),(4,2),\{6,8\} \right)$. By applying Proposition~\ref{prop10}, since $\gcd(\{6,8,4,2\})=2$, we know that $M\left((2,3),(4,2),\{6,8\} \right) = 2\cdot M\left((2,3),(2,1),\{3,4\} \right)$. Now, from Algorithm~\ref{alg11}, $M\left((2,3),(2,1),\{3,4\} \right) = \langle \{3,4\} \rangle$. Therefore, $M\left((2,3),(4,2),\{6,8\} \right) = \langle \{6,8\} \rangle = \{0,6,8,12,14,16,\ldots \}$.
	
	Since $\mathbb{N} \setminus M\left((2,3),(4,2),\{6,8\} \right)$ has infinitely many elements, its cardinality is grater than or equal to $9$ and, consequently, Theorem~\ref{thm3} assures that the problem $\mathrm{P}\left( (2,3), (4,2), \{6,8\}, 9 \right)$ has solution. Moreover, by the proof of the sufficiency of Theorem~\ref{thm3}, we have that $\{1,2,3,4,5,7,9,10,11\}$ is a solution.
\end{example}

\section{The tree associated to $\mathcal{N}(a,b,X)$}\label{tree}

A \emph{graph} $G$ is a pair $(V,E)$, where
\begin{itemize}
	\item $V$ is a non-empty set, which elements are called \emph{vertices} of $G$,
	\item $E$ is a subset of $\{(v,w) \in V \times V \mid v \neq w\}$, which elements are called \emph{edges} of $G$.
\end{itemize}
A \emph{path (of length $n$)} connecting the vertices $x$ and $y$ of $G$ is a sequence of different edges of the form $(v_0,v_1),(v_1,v_2),\ldots,(v_{n-1},v_n)$ such that $v_0=x$ and $v_n=y$.

We say that a graph $G$ is a \emph{tree} if there exists a vertex $r$ (known as the \emph{root} of $G$) such that, for every other vertex $x$ of $G$, there exists a unique path connecting $x$ and $r$. If $(x,y)$ is an edge of the tree, then we say that $x$ is a \emph{child} of $y$.

We define the graph $\mathrm{G}(\mathcal{N}(a,b,X))$ in the following way.
\begin{itemize}
	\item $\mathcal{N}(a,b,X)$ is the set of vertices of $\mathrm{G}(\mathcal{N}(a,b,X))$;
	\item $(S,S')\in \mathcal{N}(a,b,X)\times \mathcal{N}(a,b,X)$ is an edge of $\mathrm{G}(\mathcal{N}(a,b,X))$ if $S'=S\cup\{\mathrm{F}(S)\}$.
\end{itemize}
By Proposition~\ref{prop1} and \cite[Theorem 27]{variedades}, we have that $\mathrm{G}(\mathcal{N}(a,b,X))$ is a tree with root $\mathbb{N}$. Our first purpose in this section will be to establish what are the children of a vertex in such a tree. For this we need to introduce some concepts.

Let $S$ be a numerical semigroup and let $G$ be a system of generators of $S$. We say that $G$ is a \emph{minimal system of generators} of $S$ if $S\not= \langle Y \rangle$ for all $Y \subsetneq G$. It is well known (see \cite{springer}) that every numerical semigroup admits a unique minimal system of generators and that, in addition, such a system is finite. Observe that, if we denote by $\mathrm{msg}(S)$ the minimal system of generators of $S$, then $\mathrm{msg}(S) = \left( S \setminus \{0\} \right) \setminus \left( (S \setminus \{0\}) + (S \setminus \{0\}) \right)$. On the other hand, we have (see \cite{springer}) that, if $S$ is a numerical semigroup and $s \in S$, then $S \setminus \{s\}$ is another numerical semigroup if and only if $s \in \mathrm{msg}(S)$.

An immediate consequence of \cite[Proposition 24, Theorem 27]{variedades} is the next result.
\begin{theorem}\label{thm14}
	The graph $\mathrm{G}(\mathcal{N}(a,b,X))$ is a tree with root $\mathbb{N}$. Moreover, the set of children of a vertex $S \in \mathcal{N}(a,b,X)$ is
	$$\left\{ S \setminus \{m\} \mid m \in \mathrm{msg}(S), \; m > \mathrm{F}(S), \; \mbox{and} \; S \setminus \{m\} \in \mathcal{N}(a,b,X) \right\}.$$
\end{theorem} 

In the next result we will show the conditions that must satisfy $m \in \mathrm{msg}(S)$ in order to have $S\setminus \{m\} \in \mathcal{N}(a,b,X)$.

\begin{proposition}\label{prop15}
	Let $S \in \mathcal{N}(a,b,X)$ and let $m \in \mathrm{msg}(S)$. Then $S\setminus \{m\} \in \mathcal{N}(a,b,X)$ if and only if $\frac{m-b_i}{a_i} \notin S \setminus \{0\}$ for all $i \in \{1,\ldots,n\}$ and $m \notin X$.
\end{proposition}

\begin{proof}
	\emph{(Necessity.)} Since $X \subseteq S\setminus \{m\}$, we have that $m\notin X$. Let us suppose that there exists $i \in \{1,\ldots,n\}$ such that $\frac{m-b_i}{a_i} \in S \setminus \{0\}$. Since $\frac{m-b_i}{a_i} \not= m$, we have that $\frac{m-b_i}{a_i} \in S \setminus \{0,m\}$ and that $a_i \left(\frac{m-b_i}{a_i}\right) + b_i = m \not\in S \setminus \{m\}$. Therefore, $S\setminus \{m\} \notin \mathcal{N}(a,b,X)$.
	
	\emph{(Sufficiency.)} If $S\setminus \{m\} \notin \mathcal{N}(a,b,X)$, then there exists $s \in S\setminus \{0,m\}$ and there exists $i\in \{1,\ldots,n\}$ such that $a_is+b_i \notin S\setminus \{m\}$. Since $S \in \mathcal{N}(a,b,X)$, we know that $a_is+b_i \in S$. Therefore, $a_is+b_i=m$ and, consequently, $\frac{m-b_i}{a_i}=s \in S\setminus \{0\}$.
\end{proof}

Our next purpose will be to build recurrently the tree $\mathrm{G}(\mathcal{N}(a,b,X))$ from the root and joining each vertex with its children by means of edges. In order to make easy that construction, we will study the relation between the minimal system of generators of a numerical semigroup $S$ and the minimal system of generators of $S\setminus \{m\}$, where $m$ is a minimal generator of $S$ greater than $\mathrm{F}(S)$. First of all, it is clear that, if $S$ is minimally generated by $\{m,m+1,\ldots,2m-1\}$ (that is, $S=\{0,m,\to \}$), then $S\setminus \{m\} = \{0,m+1,\to \}$ is minimally generated by $\{m+1,m+2,\ldots,2m+1\}$. In other case we can apply the following result, which is \cite[Corollary 18]{frases}.

\begin{proposition}\label{prop16}
	Let $S$ be a numerical semigroup with minimal system of generators $\mathrm{msg}(S)=\left\{n_1<n_2<\cdots<n_p\right\}$. If $i \in \{2,\ldots,p\}$ and $n_i>\mathrm{F}(S)$, then
	$$\mathrm{msg}(S \setminus \{n_i\})= \left\{ \begin{array}{l}
	\left\{n_1,\ldots,n_{p}\right\} \setminus \{n_i\}, \quad \mbox{if there exists } j \in \{2,\ldots,i-1\} \\ \hspace{3.4cm} \mbox{such that } n_i+n_1-n_j \in S; \\[2mm]
	\left(\left\{n_1,\ldots,n_{p}\right\} \setminus \{n_i\}\right) \cup \{n_i+n_1\}, \quad \mbox{in other case.}
	\end{array} \right.$$
\end{proposition}

Let us illustrate the previous results with an example.

\begin{example}\label{exmp17}
	By Proposition~\ref{prop7}, it is easy to see that $S=\langle \{5,7,8,9,11\} \rangle \in \mathcal{N}((1,2),(4,1),\{5\})$. On the other hand, from Theorem~\ref{thm14}, we know that the set of children of $S$ in the tree $\mathrm{G}(\mathcal{N}((1,2),(4,1),\{5\}))$ is
	$$\left\{ S \setminus \{m\} \mid m \in \mathrm{msg}(S), \; m > \mathrm{F}(S), \; \mbox{and} \; S \setminus \{m\} \in \mathcal{N}((1,2),(4,1),\{5\}) \right\}.$$
	Since $\mathrm{F}(S)=6$, we have that $\{m\in \mathrm{msg}(S) \mid m> \mathrm{F}(S)\} = \{7,8,9,11\}$. Furthermore, by Proposition~\ref{prop15}, we know that $S\setminus \{m\} \in \mathcal{N}((1,2),(4,1),\{5\})$ if and only if $m\notin \{5\}$ and $\{m-4,\frac{m-1}{2}\} \cap (S\setminus \{0\}) = \emptyset$. Thus, since that $\{7-4,\frac{7-1}{2}\} \cap (S\setminus \{0\}) = \{8-4,\frac{8-1}{2}\} \cap (S\setminus \{0\}) = \emptyset$, $\{9-4,\frac{9-1}{2}\} \cap (S\setminus \{0\}) \not= \emptyset$, and $\{11-4,\frac{11-1}{2}\} \cap (S\setminus \{0\}) \not= \emptyset$, we conclude that $S=\langle \{5,7,8,9,11\} \rangle$ has two children. Namely, they are $\langle \{5,7,8,9,11\} \rangle \setminus \{7\} = \langle \{5,8,9,11,12\} \rangle$ and $\langle \{5,7,8,9,11\} \rangle \setminus \{8\} = \langle \{5,7,9,11,13\} \rangle$, where we have applied Proposition~\ref{prop16}.
\end{example}

Following the idea of the previous example, we can recurrently build the tree $\mathrm{G}(\mathcal{N}((1,2),(4,1),\{5\}))$ beginning with its root, that is, from ${\mathbb N}=\langle \{1\} \rangle$.

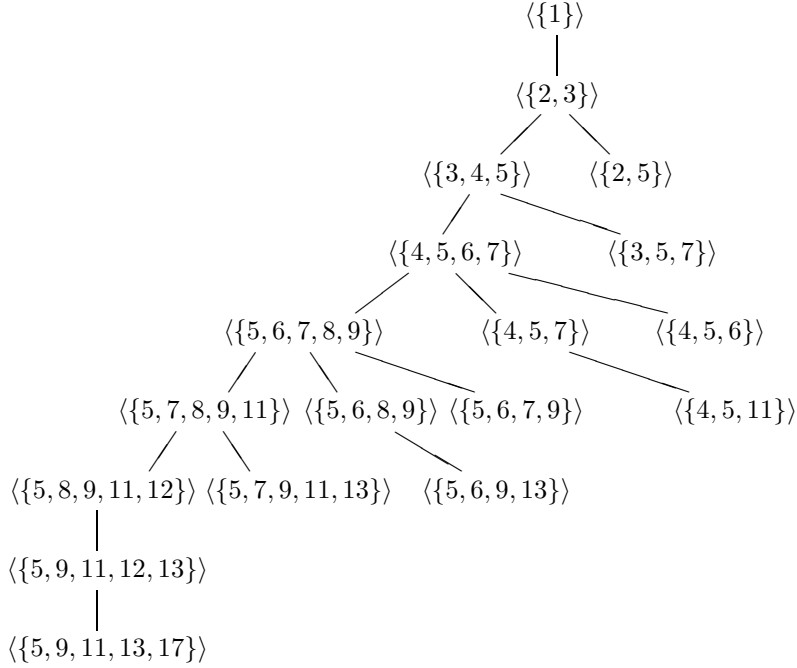
\begin{center}
	\begin{picture}(315,250)
	\put(204,240){$\langle \{1\} \rangle$}
	\put(216,235){\line(0,-1){15}}
	\put(200,210){$\langle \{2,3\} \rangle$}
	\put(210,205){\line(-1,-1){15}} \put(221,205){\line(1,-1){15}}
	\put(165,180){$\langle \{3,4,5\} \rangle$} \put(228,180){$\langle \{2,5\} \rangle$}
	\put(183,175){\line(-2,-3){10}} \put(195,175){\line(3,-1){45}}
	\put(152,150){$\langle \{4,5,6,7\} \rangle$} \put(235,150){$\langle \{3,5,7\} \rangle$}
	\put(160,145){\line(-4,-3){20}} \put(178,145){\line(1,-1){15}} \put(198,145){\line(4,-1){60}}
	\put(90,120){$\langle \{5,6,7,8,9\} \rangle$} \put(187,120){$\langle \{4,5,7\} \rangle$} \put(253,120){$\langle \{4,5,6\} \rangle$}
	\put(102,115){\line(-2,-3){10}}  \put(123,115){\line(2,-3){10}} \put(140,115){\line(3,-1){45}} \put(221,115){\line(3,-1){45}}
	\put(50,90){$\langle \{5,7,8,9,11\} \rangle$} \put(120,90){$\langle \{5,6,8,9\} \rangle$} \put(175,90){$\langle \{5,6,7,9\} \rangle$} 
	\put(260,90){$\langle \{4,5,11\} \rangle$}
	\put(72,85){\line(-2,-3){10}} \put(90,85){\line(2,-3){10}} \put(155,85){\line(5,-3){25}}
	\put(9,60){$\langle \{5,8,9,11,12\} \rangle$} \put(83,60){$\langle \{5,7,9,11,13\} \rangle$} \put(165,60){$\langle \{5,6,9,13\} \rangle$}
	\put(42,55){\line(0,-1){15}}
	\put(8,30){$\langle \{5,9,11,12,13\} \rangle$}
	\put(42,25){\line(0,-1){15}}
	\put(8,0){$\langle \{5,9,11,13,17\} \rangle$}
	\end{picture}
\end{center}

Let us observe that, since $\gcd(\{5\} \cup \{4,1\})=1$ and by Theorem~\ref{thm6}, then we know that $\mathcal{N}((1,2),(4,1),\{5\})$ is a finite Frobenius variety and, thereby, we have been able of building it completely in a finite number of steps.

Let us also observe that, if $S \in \mathcal{N}(a,b,X)$, then $\mathrm{g}(S)$ is equal to the length of the path connecting $S$ with ${\mathbb N}$ in the tree $\mathrm{G}((\mathcal{N}(a,b,X))$. Therefore, in order to build the elements of $\mathcal{N}(a,b,X)$ with a fixed genus $g$, we only need to build the elements of $\mathcal{N}(a,b,X)$ which are connected to ${\mathbb N}$ through a path of length less than or equal to $g$. Consequently, we have an algorithmic process to compute all the solutions of the problem $\mathrm{P}(a,b,X,g)$.

For instance, in the tree $\mathrm{G}(\mathcal{N}((1,2),(4,1),\{5\}))$, the numerical semigroups which are connected to ${\mathbb N}$ through a path of length $6$ are $\langle \{5,8,9,11,12\} \rangle$, $\langle \{5,7,9,11,13\} \rangle$, and $\langle \{5,6,9,13\} \rangle$. Therefore, the problem proposed in the introduction has three solutions. Namely, ${\mathbb N} \setminus \langle \{5,8,9,11,12\} \rangle = \{1,2,3,4,6,7\}$, ${\mathbb N} \setminus \langle \{5,7,9,11,13\} \rangle = \{1,2,3,4,6,8\}$, and ${\mathbb N} \setminus \langle \{5,6,9,11,13\} \rangle = \{1,2,3,4,7,8\}$.

We finish with an example in which $\mathcal{N}(a,b,X)$ is an infinite Frobenius variety.

\begin{example}\label{exmp18}
	Let us compute all the solutions of $\mathrm{P}((2,3),(4,2),\{6,8\},4)$. First of all, from Example~\ref{exmp13}, we know that $M\left((2,3),(4,2), \{6,8\} \right) = \langle \{6,8\} \rangle$ and, by Theorem~\ref{thm3}, that the problem has solution. Moreover, since $\gcd(\{6,8\} \cup \{4,2\})=2\not=1$, we have that $\mathcal{N}((2,3),(4,2),\{6,8\})$ is a infinite Frobenius variety. However, in a finite number of steps, we can compute the elements of $\mathrm{G}(\mathcal{N}((2,3),(4,2),\{6,8\}))$ which are connected to ${\mathbb N}$ through a path of length $4$, such as we show in the following scheme.
	
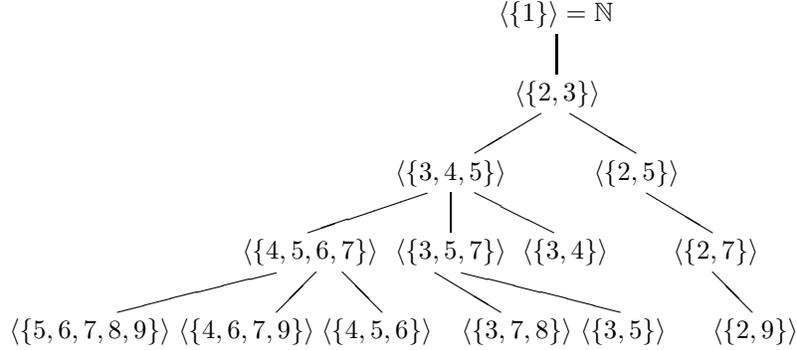
\begin{center}
		\begin{picture}(315,130)
		\put(194,120){$\langle \{1\} \rangle = {\mathbb N}$}
		\put(216,115){\line(0,-1){15}}
		\put(200,90){$\langle \{2,3\} \rangle$}
		\put(210,85){\line(-5,-3){25}} \put(221,85){\line(5,-3){25}}
		\put(155,60){$\langle \{3,4,5\} \rangle$} \put(230,60){$\langle \{2,5\} \rangle$}
		\put(167,55){\line(-3,-1){45}}  \put(176,55){\line(0,-1){15}} \put(185,55){\line(2,-1){30}} \put(250,55){\line(5,-3){25}}
		\put(97,30){$\langle \{4,5,6,7\} \rangle$} \put(155,30){$\langle \{3,5,7\} \rangle$} \put(203,30){$\langle \{3,4\} \rangle$} 
		\put(260,30){$\langle \{2,7\} \rangle$}
		\put(110,25){\line(-4,-1){60}} \put(125,25){\line(-1,-1){15}} \put(135,25){\line(1,-1){15}} 
		\put(170,25){\line(5,-3){25}} \put(180,25){\line(4,-1){60}} \put(275,25){\line(1,-1){15}}
		\put(9,0){$\langle \{5,6,7,8,9\} \rangle$} \put(73,0){$\langle \{4,6,7,9\} \rangle$} \put(127,0){$\langle \{4,5,6\} \rangle$} 
		\put(180,0){$\langle \{3,7,8\} \rangle$} \put(225,0){$\langle \{3,5\} \rangle$} \put(275,0){$\langle \{2,9\} \rangle$}
		\end{picture}
	\end{center}
	Therefore, the sets ${\mathbb N} \setminus \langle \{5,6,7,8,9\} \rangle = \{1,2,3,4\}$, ${\mathbb N} \setminus \langle \{4,6,7,9\} \rangle = \{1,2,3,5\}$, ${\mathbb N} \setminus \langle \{4,5,6\} \rangle = \{1,2,3,7\}$, ${\mathbb N} \setminus \langle \{3,7,8\} \rangle = \{1,2,4,5\}$, ${\mathbb N} \setminus \langle \{3,5\} \rangle = \{1,2,4,7\}$, and ${\mathbb N} \setminus \langle \{2,9\} \rangle = \{1,3,5,7\}$ are the (six) solutions of $\mathrm{P}((2,3),(4,2),\{6,8\},4)$.
\end{example}

\begin{remark}\label{rem-tree}
	Let us observe that, in the construction of the trees, we can assume that $X=\emptyset$ or that $a,b$ are $0$-tuples (see Remarks~\ref{rem-X} and \ref{rem-(a,b)}). Then, we obtain correctly all the possible solutions in each case. In particular, if we consider jointly such assumptions, then we get the tree associated to the full family of numerical semigroups.
\end{remark}

\section{A generalization of the  problem}\label{generalization}

Along this section $r, g$ are non-negative integers, $a=(a_1,\ldots,a_n)$ and $b=(b_1,\ldots,b_n)$ are $n$-tuples of positive integers, and $X$ is a non-empty subset of $\{r+1,\to\}$. We will denote by $\mathrm{P}_r(a,b,X,g)$ the (generalized) problem of computing all the subsets $C$ of $\{r+1,\to\}$ that fulfill the following conditions.
\begin{itemize}
	\item[] \hspace{-20pt} $(GP1)$ The cardinality of $C$ is equal to $g$.
	\item[] \hspace{-20pt} $(GP2)$ If $x,y\in \{r+1,\to\}$ and $x+y\in C$, then $C \cap \{x,y\} \neq \emptyset$.
	\item[] \hspace{-20pt} $(GP3)$ If $x \in C$ and $\frac{x-b_i}{a_i} \in \{r+1,\to\}$ for some $i\in \{1,\ldots,n\}$, then $\frac{x-b_i}{a_i} \in C$.
	\item[] \hspace{-20pt} $(GP4)$ $X \cap C = \emptyset$. 
\end{itemize}
Let us observe that $\mathrm{P}_0(a,b,X,g)=\mathrm{P}(a,b,X,g)$.

It is clear that a set $C$ is a solution of $\mathrm{P}_r(a,b,X,g)$ if and only if $S= \{0,r+1,\to\} \setminus C$ is a numerical semigroup that fulfills the following conditions.
\begin{itemize}
	\item[] \hspace{-20pt} $(GS1)$ $\mathrm{g}(S)=r+g$.
	\item[] \hspace{-20pt} $(GS2)$ If $s \in S \setminus \{0\}$, then $as+b \in S^n$.
	\item[] \hspace{-20pt} $(GS3)$ $X \subseteq S$.
\end{itemize}

Let us denote by $\mathcal{N}_r(a,b,X)$ the set of all numerical semigroups which are subsets of $\{0,r+1,\to \}$ and satisfy the conditions $(GS2)$ and $(GS3)$. Let us observe that, with this notation, the solutions of $\mathrm{P}_r(a,b,X,g)$ are the elements of the set $\left\{ \{0,r+,1\to \}\setminus S \mid S \in \mathcal{N}_r(a,b,X) \mbox{ and } \mathrm{g}(S)= r+g \right\}$. Moreover, $\mathcal{N}_r(a,b,X)=\left\{S\in \mathcal{N}(a,b,X) \mid S \subseteq \{0,r+1,\to \} \right\}$.

The following proposition is analogous to Theorem~\ref{thm3}.

\begin{proposition}\label{propA}
	Let us take $M_r(a,b,X) = \bigcap_{S \in \mathcal{N}_r(a,b,X)} S$. Then the problem $\mathrm{P}_r(a,b,X,g)$ has solution if and only if the cardinality of $\mathbb{N} \setminus M_r(a,b,X)$ is greatest than or equal to $g+r$.
\end{proposition}

\begin{proof}
	\emph{(Necessity.)} If $C$ is a solution of $\mathrm{P}_r(a,b,X,g)$, then we have that $S=\{0,r+1,\to \} \setminus C \in \mathcal{N}_r(a,b,X)$ and $\mathrm{g}(S)=g+r$. Since $M_r(a,b,X) \subseteq S$, then the cardinality of $\mathbb{N} \setminus M_r(a,b,X)$ is greatest than or equal to $g+r$.
	
	\emph{(Sufficiency.)} If $\{0,r+1,\to \} \setminus M_r(a,b,X)=\{c_1<\cdots<c_g<\cdots\}$ and $S= M_r(a,b,X) \cup \{c_g+1,\to\}$, then it is easy to see that $S \in \mathcal{N}_r(a,b,X)$ and $\mathrm{g}(S)=g+r$. Therefore, $C=\{0,r+1,\to \} \setminus S$ is a solution of $\mathrm{P}_r(a,b,X,g)$.
\end{proof}

Observe that the cardinality of $\mathbb{N} \setminus M_r(a,b,X)$ is greatest than or equal to $g+r$ if and only if the cardinality of $\{0,r+1,\to \} \setminus M_r(a,b,X)$ is greatest than or equal to $g$.

\begin{proposition}\label{propB}
	$M_r(a,b,X) = M(a,b,X)$.
\end{proposition}

\begin{proof}
	Since $\mathcal{N}_r(a,b,X) \subseteq \mathcal{N}(a,b,X)$, then $M(a,b,X) = \bigcap_{S \in \mathcal{N}(a,b,X)} S \subseteq \bigcap_{S \in \mathcal{N}_r(a,b,X)} S = M_r(a,b,X)$. Let us now see the other inclusion.
	
	Since $\{0,r+1,\to \} \in \mathcal{N}(a,b,X)$ and $\mathcal{N}(a,b,X)$ is a Frobenius variety, we have that, if $S \in \mathcal{N}(a,b,X)$, then $S \cap \{0,r+1,\to \} \in \mathcal{N}(a,b,X)$. In addition, $S \cap \{0,r+1,\to \} \subseteq \{0,r+1,\to \}$ and, therefore, $S \cap \{0,r+1,\to \} \in \mathcal{N}_r(a,b,X)$. In this way, $\mathcal{R}=\left\{S \cap \{0,r+1,\to \} \mid S \in \mathcal{N}(a,b,X) \right\} \subseteq \mathcal{N}_r(a,b,X)$. Consequently, $M_r(a,b,X) = \bigcap_{S \in \mathcal{N}_r(a,b,X)} S \subseteq \bigcap_{S \in \mathcal{R}} S = \bigcap_{S \in \mathcal{N}(a,b,X)} S = M(a,b,X)$.
\end{proof}

As an immediate consequence of Proposition~\ref{propB} and Proposition~\ref{prop10}, we have the next result.

\begin{corollary}\label{corC}
	If $\gcd\left(X \cup \{b_1,\ldots,b_n\}\right)=d$, then we get that $M_r(a,b,X) = d\cdot M\left(a,\frac{b}{d},\frac{X}{d}\right)$.
\end{corollary}

Let us observe that, as a consequence fo the previous corollary, we can use Algorithm~\ref{alg11} in order to compute $M_r(a,b,X)$.

The following result is the analogous to Theorem~\ref{thm6} for the current problem.

\begin{corollary}\label{corD}
	The following conditions are equivalent.
	\begin{enumerate}
		\item[\textit{1.}] $\mathcal{N}_r(a,b,X)$ is finite.
		\item[\textit{2.}] $M_r(a,b,X)$ is a numerical semigroup.
		\item[\textit{3.}] $\gcd(X\cup \{b_1,\ldots,b_n\})=1$.
	\end{enumerate}
\end{corollary}

\begin{proof}
	The equivalence between conditions \textit{2} and \textit{3} is consequence of Proposition~\ref{propB} and Theorem~\ref{thm6}. Now, let us see the equivalence between conditions \textit{1} and \textit{2}.
	
	$\textit{(1.} \Rightarrow \textit{2.)}$ It is enough to observe that the finite intersection of numerical semigroups is another numerical semigroup.
	
	$\textit{(2.} \Rightarrow \textit{1.)}$ If $S \in \mathcal{N}_r(a,b,X)$, then $M_r(a,b,X) \subseteq S$. Thus, $S=M_r(a,b,X) \cup Y$ for some $Y\subseteq \mathbb{N} \setminus M_r(a,b,X)$. Since $\mathbb{N} \setminus M_r(a,b,X)$ is finite, then we can conclude that $\mathcal{N}_r(a,b,X)$ is finite.
\end{proof}

Let us illustrate the previous results with several examples.

\begin{example}\label{exmpE1}
	Let us see that $\mathrm{P}_r\left( (1,2), (4,1), \{5\}, 6 \right)$ has solution if and only if $r\in\{0,1,2\}$. Since $\{5\} \subseteq \{r+1,\to \}$, then we have that $r\in\{0,1,2,3,4,\}$. From Proposition~\ref{propB} and Example~\ref{exmp12}, we have that $M=M_r\left( (1,2), (4,1), \{5\}, 6 \right) = \langle \{5,9,11,13,17\} \rangle$. Since ${\mathbb N} \setminus M = \{1,2,3,4,6,7,8,12\}$ has cardinality equal to $8$, by applying Proposition~\ref{propA}, we easily deduce that $\mathrm{P}_r\left( (1,2), (4,1), \{5\}, 6 \right)$ has solution if and only if $r\in\{0,1,2\}$.
\end{example}

\begin{example}\label{exmpE2}
	If $r \in \{0,1,2,3,4,5\}$, then we have that $\mathcal{N}_r\left( (2,3), (4,2), \{6,8\} \right)$ is an infinite set. In fact, this is an immediate consequence of Corollary~\ref{corD} and that $\gcd(\{4,2,6,8\})=2\not=1$.
\end{example}

\begin{example}\label{exmpE3}
	Let us compute $\mathrm{P}_3\left( (2,3), (4,2), \{6,8\}, 9 \right)$. From Proposition~\ref{propB} and Example~\ref{exmp13} we have that $M_3\left( (2,3), (4,2), \{6,8\} \right) = \langle \{6,8\} \rangle$. Now, if we apply the construction given in the sufficiency of Proposition~\ref{propA}, we have that $\{4,5,7,9,10,11,13,15,17\}$  is a solution.
\end{example}

If $G$ is a tree and $u,v$ are two vertices of $G$ such that there exists a path between them, then we will say that $u$ is a \emph{descendant} of $v$. The next result has an easy proof.

\begin{proposition}\label{propF}
	$\mathcal{N}_r(a,b,X)$ is the set of all descendants of $\{0,r+1,\to \}$ in the tree $G(\mathcal{N}(a,b,X))$.
\end{proposition}

A Frobenius pseudo-variety (see \cite{pseudovar}) is a non-empty family $\mathcal{P}$ of numerical semigroups that fulfills the following conditions. 
\begin{enumerate}
	\item[] \hspace{-20pt} $(PV1)$ $\mathcal{P}$ has a maximum element $\max(\mathcal{P})$ (with respect to the inclusion order).
	\item[] \hspace{-20pt} $(PV2)$ If $S,T \in \mathcal{P}$, then $S\cap T \in \mathcal{P}$.
	\item[] \hspace{-20pt} $(PV3)$ If $S \in \mathcal{P}$ and $S \not= \max(\mathcal{P})$, then $S \cup \mathrm{F}(S) \in \mathcal{P}$.
\end{enumerate}
As an immediate consequence of Proposition~\ref{propF} and the comment above to Example 7 in \cite{pseudovar}, we have the following result.

\begin{proposition}\label{propG}
	$\mathcal{N}_r(a,b,X)$ is a Frobenius pseudo-variety.
\end{proposition}

Let us observe that, if $r\geq1$, then $\max\left(\mathcal{N}_r(a,b,X)\right) = \{0,r+1,\to \} \not= {\mathbb N}$. Therefore, by applying \cite[Proposition 1]{pseudovar}, we have that $\mathcal{N}_r(a,b,X)$ is not a Frobenius variety.

Now, let us notice that the subgraph, of a tree, which is formed by a vertex and all its descendants is also a tree. We will denote by $\mathrm{G}\left(\mathcal{N}_r(a,b,X)\right)$ the subtree of $\mathrm{G}\left(\mathcal{N}(a,b,X)\right)$ formed by $\{0,r+1,\to \}$ and all its descendants.

\begin{example}\label{exmpH}
	Since the root of the tree $\mathrm{G}(\mathcal{N}_3((1,2),(4,1),\{5\}))$ is $\{0,4,\to\}$ $= \langle \{4,5,6,7\} \rangle$, from Example~\ref{exmp17} we have that such a tree is the following one.
	
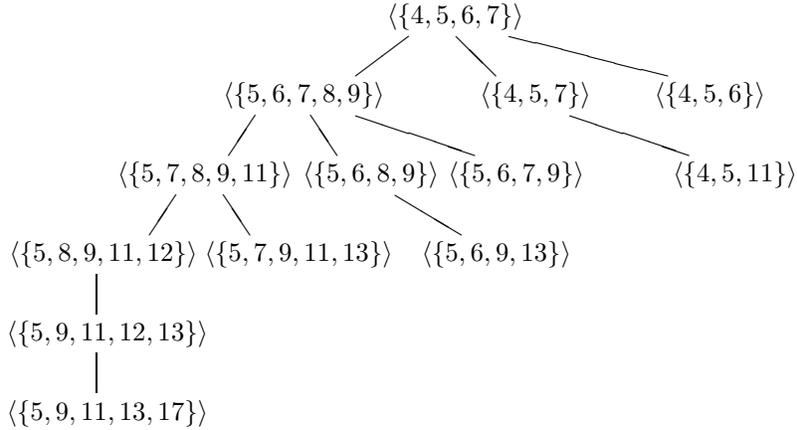
\begin{center}
		\begin{picture}(315,160)
		\put(152,150){$\langle \{4,5,6,7\} \rangle$}
		\put(160,145){\line(-4,-3){20}} \put(178,145){\line(1,-1){15}} \put(198,145){\line(4,-1){60}}
		\put(90,120){$\langle \{5,6,7,8,9\} \rangle$} \put(187,120){$\langle \{4,5,7\} \rangle$} \put(253,120){$\langle \{4,5,6\} \rangle$}
		\put(102,115){\line(-2,-3){10}}  \put(123,115){\line(2,-3){10}} \put(140,115){\line(3,-1){45}} \put(221,115){\line(3,-1){45}}
		\put(50,90){$\langle \{5,7,8,9,11\} \rangle$} \put(120,90){$\langle \{5,6,8,9\} \rangle$} \put(175,90){$\langle \{5,6,7,9\} \rangle$} 
		\put(260,90){$\langle \{4,5,11\} \rangle$}
		\put(72,85){\line(-2,-3){10}} \put(90,85){\line(2,-3){10}} \put(155,85){\line(5,-3){25}}
		\put(9,60){$\langle \{5,8,9,11,12\} \rangle$} \put(83,60){$\langle \{5,7,9,11,13\} \rangle$} \put(165,60){$\langle \{5,6,9,13\} \rangle$}
		\put(42,55){\line(0,-1){15}}
		\put(8,30){$\langle \{5,9,11,12,13\} \rangle$}
		\put(42,25){\line(0,-1){15}}
		\put(8,0){$\langle \{5,9,11,13,17\} \rangle$}
		\end{picture}
	\end{center}
 \end{example}

Let us observe that $\mathcal{N}_r(a,b,X)$ is the set of vertices in $\mathrm{G}\left(\mathcal{N}_r(a,b,X)\right)$, and that $(S,S') \in \mathcal{N}_r(a,b,X) \times \mathcal{N}_r(a,b,X)$ is an edge of $\mathrm{G}\left(\mathcal{N}_r(a,b,X)\right)$ if and only if $S'=S\cup\{\mathrm{F}(S)\}$. It is also clear that, if $S \in \mathcal{N}_r(a,b,X)$, then the set formed by the children of $S$ in $\mathcal{N}_r(a,b,X)$ is the same that the set formed by the children of $S$ in $\mathcal{N}(a,b,X)$. In this way, by applying Theorem~\ref{thm14}, we have the next result.

\begin{proposition}\label{propI}
	The graph $\mathrm{G}\left(\mathcal{N}_r(a,b,X)\right)$ is a tree with root $\{0,r+1,\to \}$. Moreover, the set of children of a vertex $S$ in $\mathrm{G}\left(\mathcal{N}_r(a,b,X)\right)$ is
	$$\left\{ S \setminus \{m\} \mid m \in \mathrm{msg}(S), \; m > \mathrm{F}(S), \; \mbox{and} \; S \setminus \{m\} \in \mathcal{N}(a,b,X) \right\}.$$
\end{proposition}

Now, let us notice that, by using Propositions~\ref{prop15} and \ref{prop16}, we can compute the children of any vertex $S$ in $\mathrm{G}\left(\mathcal{N}_r(a,b,X)\right)$ and, consequently, we have an algorithmic process to recurrently build the elements of $\mathcal{N}_r(a,b,X)$.

We finish this section with an illustrative example about the above comment.

\begin{example}\label{exmpJ}
	Let us compute all the solutions of $\mathrm{P}_3((2,3),(4,2),\{6,8\},4)$. In order to do this, we have to determine the vertices of $\mathrm{G}(\mathcal{N}_3((2,3),(4,2),\{6,8\}))$ which are connected to $\{0,4,\to\} = \langle \{4,5,6,7\} \rangle$ through a path of length $4$. 
	
	Let us observe that, if $A$ is the set of vertices (of a tree) which are connected to the root through a path of length $k$, then the set formed by all vertices that are children of some vertex of $A$ is just the set of vertices which are connected to the root through a path of length $k+1$. Thus, if we denote by $A_i$ the set formed by the vertices of $\mathrm{G}(\mathcal{N}_3((2,3),(4,2),\{6,8\}))$ which are connected to $\langle \{4,5,6,7\} \rangle$ through a path of length $i$, then (by applying Propositions~\ref{propI}, \ref{prop15}, and \ref{prop16}) we obtain recurrently the following sets.
	\begin{itemize}
		\item $A_0 = \{ \langle \{4,5,6,7\} \rangle \}$
		\item $A_1 = \{ \langle \{5,6,7,8,9\} \rangle, \langle \{4,6,7,9\} \rangle, \langle \{4,5,6\} \rangle \}$
		\item $A_2 =\{ \langle \{6,7,8,9,10,11\} \rangle, \langle \{5,6,8,9\} \rangle, \langle \{5,6,7,8\} \rangle, \langle \{4,6,9,11\} \rangle, \\ \mbox{ } \hspace{9.5mm} \langle \{4,6,7\} \rangle \}$
		\item $A_3 =\{ \langle \{6,8,9,10,11,13\} \rangle, \langle \{6,7,8,10,11\} \rangle, \langle \{6,7,8,9,11\} \rangle, \\ \mbox{ } \hspace{9.5mm} \langle \{6,7,8,9,10\} \rangle, \langle \{5,6,8\} \rangle, \langle \{4,6,11,13\} \rangle, \langle \{4,6,9\} \rangle \}$
		\item $A_4 =\{ \langle \{6,8,10,11,13,15\} \rangle, \langle \{6,8,9,11,13\} \rangle, \langle \{6,8,9,10,13\} \rangle, \\ \mbox{ } \hspace{9.5mm} \langle \{6,8,9,10,11\} \rangle, \langle \{6,7,8,11\} \rangle, \langle \{6,7,8,10\} \rangle, \langle \{6,7,8,9\} \rangle, \\ \mbox{ } \hspace{9.5mm} \langle \{4,6,13,15\} \rangle, \langle \{4,6,11\} \rangle \}$
	\end{itemize}
	Therefore, the set of solutions of $\mathrm{P}_3((2,3),(4,2),\{6,8\},4)$ is
	$$\left\{ \langle \{4,5,6,7\} \rangle \setminus S \mid S \in A_4 \right\} = \left\{ \{4,5,7,9\}, \{4,5,7,10\}, \{4,5,7,11\}, \{4,5,7,13\}, \right.$$
	$$\left. \{4,5,9,10\}, \{4,5,9,11\}, \{4,5,10,11\}, \{5,7,9,11\}, \{5,7,9,13\} \right\}.$$

\end{example}


\end{document}